\documentclass[10pt, twoside, leqno]{article}
\usepackage{amsmath,amsthm}
\usepackage{amsfonts}

\newtheorem{theorem}{Theorem}[section]
\newtheorem{lemma}[theorem]{Lemma}
\theoremstyle{definition}
\newtheorem{definition}[theorem]{Definition}

\numberwithin{equation}{section}

\begin{document}


\title{Difficulties of the set of natural numbers}

\author{Qiu Kui Zhang\\
Nanjing University of Information Science and Technology\\
210044 Nanjing, China\\
E-mail: zhangqk@nuist.edu.cn}

\date{}
\maketitle


\renewcommand{\thefootnote}{}



\footnote{
\begin{flushright}
Manuscript submitted to arXiv
\end{flushright}
}

\renewcommand{\thefootnote}{\arabic{footnote}}
\setcounter{footnote}{0}

\begin{abstract}
In this article some difficulties are deduced from the set of natural numbers. By using the method of transfinite recursion we define an iterative process which is designed to deduct all the non-greatest elements of the set of natural numbers. But unexpectedly we meet some difficulties in answering the question of whether the iterative process can deduct all the elements of the set of natural numbers. The demonstrated difficulties suggest that if we regard the class of natural numbers as a set we will be confronted with either a contradiction or a conflict with the axiom of regularity. As a result, we have the conclusion that the class of natural numbers is not a set but a proper class.
\end{abstract}


\section*{Introduction}
That all the natural numbers can be pooled together to form an infinite set is a fundamental hypothesis in mathematics and philosophy, which now is widely accepted by mathematicians and scientists from various disciplines. With this hypothesis mathematicians had systematically developed a theory of infinity, namely, set theory which had become the foundation of modern mathematics and science ever since. Although once this hypothesis was a controversial issue between different schools of mathematics and philosophy and some intuitionists object to it on the grounds that a collection of objects produced by an infinite process should not be treated as a completed entity \cite{A}, they do not provide further evidence to prove that it will cause logical contradiction. And no contradiction resulting from this hypothesis had ever been reported. Today the debate has subsided and most scientists do not doubt about the validity of this hypothesis. However, in our recent study we have found some logical contradictions resulting from this hypothesis, which suggest if the axiom of infinity holds it either leads to a contradiction or a conflict with the axiom of regularity. So set theory is not as consistent as we had thought before. We anticipate our study to be a starting point for the establishment of a more sophisticated foundation theory to prevent mathematics and thus other sciences from contradiction.

\section{The definition of natural numbers in set theory}
In order to define natural numbers and study the set of natural numbers within the framework of set theory it is necessary to define a successor relation first \cite{A}.
\begin{definition}
The successor of a set $x$ is the set $x^+=x \cup \{x\}$.
\end{definition}
The notation `+' in above definition represents the successor operator which can be applied to any set to obtain its successor.

In set theory the first natural number 0 is defined with the empty set $\phi$, then number 1 with the successor of 0 and so on. To make the expression more intuitively we usually use the more suggestive notation $n+1$ for $n^+$ when $n$ is a number. So we have following inductive definition of natural numbers \cite{A}
\begin{definition}
\label{def0}
The definition of natural numbers
\begin{enumerate}
\item \label{item1} $0=\phi$ is a natural number.
\item \label{item2} If $n$ is a natural number, its successor $n+1$ is also a natural number.
\item \label{item3} All natural numbers are obtained by application of \ref{item1} and \ref{item2}.
\end{enumerate}
\end{definition}

The first rule of above definition is the basis of the induction which defines the initial natural number 0, and the second rule is the inductive step which can be repeatedly applied to obtain other natural numbers. The third rule is the restriction clause. So we can assign each natural number a certain value of a particular set
\begin{equation*}
0=\phi,1=\phi^+,2=\phi^{++}, \cdots
\end{equation*}

Whether all the natural numbers can be pooled together to form a completed infinite entity i.e. a set is a critical issue in mathematics and philosophy. Around it two opposite concepts of infinity have been developed, which are potential infinity and actual infinity. The former regards the infinite series 0, 1, 2, ... is potentially endless and the process of adding more and more numbers cannot be exhausted in principle, so it never can make a definite entity. The latter is based one the hypothesis that all natural numbers can form an actual, completed totality, namely, a set. That means the static set has already been completed and contained all natural numbers. Set theory is based on the notion of actual infinity that is clearly manifested in the axiom of infinity which postulates the existence of an inductive set and thus guarantees the existence of the set of natural numbers.

\section{Difficulties of the set of natural numbers}
In set theory the axiom of infinity which postulates the existence of an inductive set guarantees the existence of the set of natural numbers.\\
\textbf{The Axiom of Infinity}. An inductive set exists \cite{A}.

Because $\mathbb{N}$, the set of natural numbers, is the smallest inductive set, it is easy to prove its existence based on the axiom of infinity. Let $C$ be an existing inductive set; then we justify the existence of $\mathbb{N}$ on the basis of the axiom of comprehension \cite{A}
\begin{equation*}
\mathbb{N}=\{x \in C|x\in I~\mbox{for every inductive set} ~ I\}.
\end{equation*}
That implies if $C$ exists then $\mathbb{N}$ exists.

Usually set $\mathbb{N}$ can be expressed as an infinite list of natural numbers such as
\begin{equation}
\label{eq1}
\mathbb{N}=\{0,1,2,\cdots\}.
\end{equation}
or briefly as
\begin{equation}
\label{eq2}
\mathbb{N}=\{x|n(x)\}.
\end{equation}
where $n(x)$ is the predicate that $x$ is a natural number. However, this form of expression obviously uses the comprehension principle, which is thought to be the source of paradoxes in Cantor's naive set theory. Whether the using of comprehension principle here will result in contradiction is an interesting issue to us. And it is indeed the case, for we have found sufficient evidence to prove that the notion of the set of natural numbers is illogical and will lead to logical contradiction. Here we show our findings of a sequence of conflicts based on the question whether there is the greatest element in set $\mathbb{N}$. First let's consider a special collection $S$ of all $x\in \mathbb{N}$ with the property $P(x)$
\begin{equation}
\label{eq3}
S=\{x\in \mathbb{N}|P(x)\}.
\end{equation}
where the property $P(x)$ is $\forall y\in \mathbb{N}(x\ge y)$ which means $x$ is greater than or equal to all the elements of $\mathbb{N}$. Here the relation $x\ge y$ can be interpreted as the set relation $y\in x \vee y=x$ when $x$ and $y$ are ordinal numbers. According to the axiom schema of comprehension \cite{A}, if $\mathbb{N}$ is a set, $S$ is a definite set. Obviously, if $S$ is an empty set the greatest element of $\mathbb{N}$ does not exist; if $S$ is not an empty set it must contain the greatest element of $\mathbb{N}$ and thus the greatest element of $\mathbb{N}$ does exist. According to the law of excluded middle, for all $x$ of $\mathbb{N}$, $x$ either has or does not have the property $P(x)$, so intuitively we have following method to obtain set $S$. That is we can deduct all $\mathbb{N}$'s elements without the property $P$ from $\mathbb{N}$ and the remaining part should be $S$. To do this we need to define an iterative process with transfinite recursion to recursively deduct all non-greatest elements of $\mathbb{N}$. As transfinite recursion can go into transfinite steps, it has the ability to deduct infinitely many non-greatest elements. Therefore it is feasible to use transfinite recursion to deduct all non-greatest elements and obtain the particular remaining part. The iterative process can be implemented in this way. Choose two elements out of $\mathbb{N}$, remove the smaller one that clearly does not have the property $P$ and return the bigger one to the remaining part. Repeat this procedure until there are no two elements left in the remaining part that can be further chosen to implement further deduction and this particular remaining part should be $S$. So whether the iterative process can deduct all the elements of $\mathbb{N}$ is a critical question. If it can the remaining part is empty, so the the greatest element does not exist; if it cannot the remaining part is not empty, so the greatest element maybe exists. 
To clarify the fact we make in-depth investigation by translating the question into a well-defined mathematical representation. First let's define a $Min$ function applied to two natural numbers to obtain the smaller one
\begin{equation*}
Min(x,y)=\left\{
\begin{array}
{r@{\quad if\quad}l}
x & x \le y \\y & y < x
\end{array} \right.
\end{equation*}
In set theory, it is obvious that the $Min(x,y)$ function can be implemented as the intersection of natural numbers $x$ and $y$
\begin{equation}
\label{eq4}
Min(x,y)=x\cap y.
\end{equation}
Then according to the axiom of choice \cite{A}, there is a choice function $f$, defined on set $X=P(\mathbb{N})\backslash\{\phi\}$ (where $P(\mathbb{N})$ is the power set of $\mathbb{N}$, and $P(\mathbb{N})\backslash\{\phi\}$ represents the set-theoretic difference of $P(\mathbb{N})$ and $\{\phi\}$), such that
\begin{equation*}
\forall x(x \in X \to f(x)\in x)
\end{equation*}
where symbol $\to$ symbolizes the the relation of material implication. So we have following inductive definition.

\begin{definition}
\label{def1}
For all ordinals $\alpha\in On$, recursively define following transfinite sequences $A_\alpha$, $B_\alpha$ and $a_\alpha$.
\begin{enumerate}
\item $A_\alpha=\{a_\beta|\beta<\alpha\}$.
\item $B_\alpha=\mathbb{N}\backslash A_\alpha$.
\item
$
a_\alpha=\left\{
\begin{array}
{c@{\quad if\quad}l}
Min(f(B_\alpha),f(B_\alpha\backslash\{f(B_\alpha)\})) & Card(B_\alpha) > 1 \\b & Card(B_\alpha) = 1
\\c & Card(B_\alpha) =0
\end{array} \right.
$.

\end{enumerate}
\end{definition}

Where ordinal number $\alpha$ indicates a particular recursion step, $A_\alpha$ is the set of all the elements that have already been deducted from $\mathbb{N}$ before step $\alpha$ is performed, $B_\alpha$ is the particular remaining part of $\mathbb{N}$ exactly before step $\alpha$ is performed ($B_\alpha$ also can be understood as the particular remaining part of $\mathbb{N}$ exactly after all steps before step $\alpha$ have been performed
), $a_\alpha$ is the particular element of $\mathbb{N}$ that is deducted at the current step $\alpha$ if $B_\alpha$ still contains more than one element otherwise it equals $b$ or $c$, $Card(B_\alpha)$ stands for the cardinality of set $B_\alpha$, $b=\{2\}$ and $c=\{2,3\}$ are sets not belong to $\mathbb{N}$.

It is easy to obtain every elements of the transfinite sequences $A_\alpha$, $B_\alpha$ and $a_\alpha$ with definition \ref{def1}. First it is obvious that $A_0=\phi$ (before step 0 is performed nothing is deducted), $B_0=\mathbb{N}$ (before step 0 is performed the remaining part is exact $\mathbb{N}$) and $a_0=Min(f(\mathbb{N}),f(\mathbb{N}\backslash\{f(\mathbb{N})\}))$. Second if we have obtained all $a_\beta$ for $\beta<\alpha$, then we can obtain $A_\alpha$, $B_\alpha$ and $a_\alpha$ with the three clauses of definition \ref{def1} respectively. So, in line with the principle of transfinite recursion, the transfinite sequences $A_\alpha$, $B_\alpha$ and $a_\alpha$ exist.

It is obvious that we can determine whether the iterative process can deduct all the elements of $\mathbb{N}$ by the value of sequence $B_\alpha$. If and only if fore every ordinal step $\alpha$ we have $B_\alpha\ne\phi$, then the iterative process cannot deduct all the elements of $\mathbb{N}$. So the necessary and sufficient condition for the iterative process cannot deduct all the elements of $\mathbb{N}$ is $\forall \alpha(B_\alpha\ne\phi)$. As a result, we have following definitions.

\begin{definition}
\label{def2}
We say the iterative process cannot deduct all the elements of $\mathbb{N}$ by step $\alpha$ if and only if $B_{\alpha+1}\not=\phi$ (the remaining part is still not empty after step $\alpha$ is performed).
\end{definition}
And then we have definition \ref{def3}.
\begin{definition}
\label{def3}
We say the the iterative process cannot deduct all the elements of $\mathbb{N}$ if and only if the iterative process cannot deduct all the elements of $\mathbb{N}$ by every ordinal step, which can be written as the formula $\forall\alpha(B_{\alpha+1}\neq\phi)$.
\end{definition}

Considering $\forall \alpha(B_\alpha\ne\phi)\implies\forall \alpha(B_{\alpha+1}\ne\phi)$ and $\forall\alpha(B_{\alpha+1}\ne\phi)\implies\forall\alpha(B_\alpha\ne\phi)$ (observe $B_{\alpha+1}\subseteq B_\alpha$, the first property of lemma \ref{lemma0}), definition \ref{def3} is obviously reasonable. Based on above terminologies, the expression that the iterative process cannot deduct all the elements of $\mathbb{N}$ before step $\beta$, which restricts its concerning domain to the steps before step $\beta$ while the statement made in \ref{def3} is about the whole domain of all ordinal steps (in other words, the statement in \ref{def3} refers to the iterative process cannot deduct all the elements in the whole domain of all ordinal steps while the above expression refers to the iterative process cannot only in the restricted domain of all steps before $\beta$), should be logically interpreted as the iterative process cannot deduct all the elements of $\mathbb{N}$ by every step before step $\beta$ or formally as $\forall \alpha(\alpha<\beta\to B_{\alpha+1}\neq\phi)$.

With above definitions and interpretations we expect to answer the question whether the iterative process can deduct all the elements of $\mathbb{N}$. But in the following study we meet some difficulties in answering the question. If the answer is yes we will encounter a contradiction; if the answer is no the greatest natural number must exist and will lead to a conflict with the axiom of regularity, another contradiction. The following sections show the dilemma of how to answer the question.

According to the clause 3 of definition \ref{def1}, the recursion steps can be classified into three classes corresponding to the three conditions $Card(B_\alpha)>1$, $Card(B_\alpha)=1$ and $Card(B_\alpha) = 0$. And it is easy to see that only if the step $\alpha$ satisfies the first condition $Card(B_\alpha)>1$ does the iterative process deduct one element from $\mathbb{N}$ at step $\alpha$; otherwise it deducts nothing from $\mathbb{N}$ at step $\alpha$. The second and third conditions are end conditions which mean once the recursion step has meets these conditions the deduction of elements ends and the remaining part keep invariant after that step. The third condition $Card(B_\alpha)=0$ is the empty end condition which implies the iterative process can deduct all the elements of $\mathbb{N}$; the second condition $Card(B_\alpha)=1$ is the non-empty end condition which implies the iterative process cannot deduct all the elements of $\mathbb{N}$. Then we have lemma \ref{lemma0}.
\begin{lemma}
\label{lemma0}
The transfinite sequences have following properties.
\begin{enumerate}
\item \label{p1} $\beta \le \alpha \rightarrow A_\beta \subseteq A_\alpha \wedge B_\alpha \subseteq B_\beta$.
\item \label{p2} $\alpha \ne \beta \wedge Card(B_\alpha)>1 \wedge Card(B_\beta>1)\to B_\alpha \ne B_\beta$.
\item \label{p3} $\exists \gamma (Card(B_\gamma)=1)$.
\end{enumerate}
\end{lemma}

\begin{proof}
\begin{enumerate}
\item Notice $\beta \le \alpha$ and $A_\alpha$'s definition. Then for any $x$, we have\\
    $x \in A_\beta \implies \exists \gamma(a_\gamma=x\wedge \gamma<\beta)\implies \exists \gamma(a_\gamma=x\wedge \gamma<\alpha)\implies x\in A_\alpha$\\
    where symbol $\implies$ symbolizes the relation of logical consequence. So we have\\
    $A_\beta \subseteq A_\alpha$ \\
    and\\
    $A_\beta \subseteq A_\alpha \implies \forall x(x\notin A_\alpha \to x\notin A_\beta)\\
    \implies \forall x(x\in \mathbb{N}\wedge x\notin A_\alpha \to x\in \mathbb{N}\wedge x\notin A_\beta)\\
    \implies \forall x(x\in \mathbb{N}\backslash A_\alpha \to x\in \mathbb{N}\backslash A_\beta)\\
    \implies \forall x(x\in B_\alpha \to x\in B_\beta)\\
    \implies B_\alpha \subseteq B_\beta$\\
    So we obtain property \ref{p1}.
\item If $\alpha \ne \beta$, then either $\alpha < \beta$ or $\beta<\alpha$. Let $\alpha < \beta$. Then\\
    $Card(B_\alpha)>1 \implies a_\alpha=Min(f(B_\alpha),f(B_\alpha\backslash\{f(B_\alpha)\}))\implies a_\alpha \in B_\alpha$\\
    Then noticing $A_\alpha$'s definition and $\alpha < \beta\implies \alpha+1 \le \beta\implies B_\beta \subseteq B_{\alpha+1}$, we have\\
    $a_\alpha \in A_{\alpha+1} \implies a_\alpha \notin \mathbb{N}\backslash A_{\alpha+1}\implies a_\alpha \notin B_{\alpha+1} \implies a_\alpha \notin B_\beta$\\
    So considering above two cases: $a_\alpha \in B_\alpha$ and $a_\alpha \notin B_\beta$, we obtain\\
    $B_\alpha\ne B_\beta$\\
    For the same reason it is easy to prove if $\beta<\alpha$ then $B_\beta \ne B_\alpha$, so we have property \ref{p2} which indicates all $B_\alpha$ in the transfinite sequence are non-repeating when they satisfy $Card(B_\alpha)>1$.

\item Let $B=\{B_\alpha|Card(B_\alpha)>1\}$, so all the members of $B$ are subsets of $\mathbb{N}$. Therefore $B$ is a subset of $P(N)$ that implies $B$ is a set.\\
    Then let $A=\{\alpha|Card(B_\alpha)>1\}$. From property \ref{p2} of lemma \ref{lemma0} we know all $B_\alpha$ in the sequence are non-repeating when they satisfy $Card(B_\alpha)>1$, so there is a one-to-one correspondence, $F: A\leftrightarrow B$ (ordinal $\alpha$ corresponds to $B_\alpha$), between $A$ and $B$. So $A$ is a set also, or precisely it is a set of some ordinals. Then for any ordinal numbers $\alpha$ and $\beta$, we have following logical derivation\\
    $\alpha \in A \wedge\beta<\alpha \implies Card(B_\alpha)>1\wedge B_\alpha\subseteq B_\beta \implies Card(B_\beta)>1\implies \beta\in A$\\
    As a result, we obtain\\
    $\forall \alpha \forall \beta(\beta<\alpha \wedge \alpha \in A \to\beta\in A)$\\
    That indicates set $A$ is an initial segment of ordinal, so there is an ordinal number $\lambda$ equals $A$\\
    $A=\lambda=\{\alpha|\alpha<\lambda\}$\\
    Observing the axiom of regularity, we have $\lambda\notin\lambda$ and thus $\lambda\notin A$. That implies ordinal number $\lambda$ must not satisfy set $A$'s condition, so $Card(B_\lambda)\not>1$. Therefore there are only two cases, i.e., $Card(B_\lambda)=0$ or $Card(B_\lambda)=1$. The first case $Card(B_\lambda)=0$, which means the recursion meets the empty end condition, implies the iterative process can deduct all the elements of $\mathbb{N}$. The second case $Card(B_\lambda)=1$, which means the recursion meets the non-empty end condition, implies the iterative process cannot deduct all the elements of $\mathbb{N}$. Let $Card(B_\lambda)=0$. Then\\
    $B_\lambda=\phi$\\
    As a result, we have\\
    $B_{\lambda+1}=\phi$\\
    So after step $\lambda$ is performed the iterative process has already deducted all the elements of $\mathbb{N}$. Observe $Card(B_\lambda)=0$; we know the iterative process deducts nothing from $\mathbb{N}$ at step $\lambda$. So, before step $\lambda$ is performed the iterative process has already deducted all the elements of $\mathbb{N}$ also. Therefore the iterative process can deduct all the elements of $\mathbb{N}$ before step $\lambda$.\\
    On the other hand, notice that all the ordinal numbers $\alpha$ less than $\lambda$ are $\lambda$'s members; then for any ordinal $\alpha$ we have\\
    $\alpha<\lambda \implies \alpha\in\lambda \implies \alpha\in A\implies Card(B_\alpha)>1\implies Card(B_\alpha\backslash\{a_\alpha\})>0 \implies Card(B_{\alpha+1})>0\implies B_{\alpha+1}\ne \phi$\\
    So we obtain\\
    $\forall \alpha(\alpha<\lambda \to B_{\alpha+1}\ne \phi)$\\
    that indicates by every step before $\lambda$ the iterative process cannot deduct all the elements of $\mathbb{N}$. So the iterative process cannot deduct all the elements of $\mathbb{N}$ before step $\lambda$ that contradicts the previous conclusion. As a result, to prevent this obvious contradiction, the assumption $Card(B_\lambda)=0$ must be invalid, so $Card(B_\lambda)=1$. So we have obtained a particular ordinal number $\lambda$ satisfies $Card(B_\lambda)=1$.Therefore we obtain property \ref{p3}.
\end{enumerate}
\end{proof}

Notice that every non-empty subset $x$ of $\mathbb{N}$ has its least element. Let the choice function $f(x)$ choose the least element of $x$. So that
\begin{equation}
\label{eq5}
f(x)=\cap x.
\end{equation}
and the equation in the clause 3 of definition \ref{def1} becomes
\begin{equation}
\label{eq6}
a_\alpha=\left\{
\begin{array}
{c@{\quad if\quad}l}
\cap B_\alpha & Card(B_\alpha) > 1 \\b & Card(B_\alpha) = 1
\\c & Card(B_\alpha) =0
\end{array} \right..
\end{equation}
\begin{proof}
If $Card(B_\alpha)>1$, then\\
$
a_\alpha=Min(f(B_\alpha),f(B_\alpha\backslash\{f(B_\alpha)\}))
$\\
Observe Eq. (\ref{eq4}) and (\ref{eq5}). Then we have\\
$
a_\alpha=f(B_\alpha)\cap f(B_\alpha\backslash\{f(B_\alpha)\})\\
=(\cap B_\alpha) \cap (\cap (B_\alpha\backslash\{f(B_\alpha)\}))\\
=(\cap B_\alpha) \cap (\cap (B_\alpha\backslash\{\cap(B_\alpha)\}))\\
=\cap B_\alpha
$\\
So we obtain Eq. (\ref{eq6}). From it we know only under condition $Card(B_\alpha)>1$ does the recursion step generate $a_\alpha$ belongs to $\mathbb{N}$, so if $a_\alpha$ belongs to $\mathbb{N}$ it must be generated by the first case of Eq. (\ref{eq6}). Therefore we have
\begin{equation}
\label{eq7}
a_\alpha\in \mathbb{N} \to a_\alpha=\cap B_\alpha.
\end{equation}
\end{proof}
And the transfinite sequences have the additional property
\begin{equation}
\label{eq8}
\forall x(x\in A_\alpha \cap \mathbb{N} \wedge B_\alpha \neq \phi \to x\le\cap B_\alpha).
\end{equation}
\begin{proof}
Let $B_\alpha\neq \phi$ and $\beta<\alpha$, then from property \ref{p1} of lemma we know both $B_\alpha$ and $B_\beta$ are non-empty sets of natural numbers and $B_\alpha\subseteq B_\beta$. So\\
$\cap B_\beta\le \cap B_\alpha$\\
Above derivation can be expressed as formula (\ref{eq9}) to facilitate following derivation
\begin{equation}
\label{eq9}
B_\alpha\neq \phi \wedge \beta<\alpha\to \cap B_\beta\le \cap B_\alpha.
\end{equation}
Observe formula (\ref{eq7}) and (\ref{eq9}). Then for any $x$, we have\\
$x\in A_\alpha \cap \mathbb{N} \wedge B_\alpha\ne \phi\\
\implies x\in A_\alpha \wedge x\in \mathbb{N} \wedge B_\alpha\ne \phi\\
\implies\exists \beta( \beta<\alpha \wedge a_\beta=x) \wedge x\in \mathbb{N} \wedge B_\alpha\ne \phi\\
\implies\exists \beta( \beta<\alpha \wedge a_\beta=x \wedge x \in \mathbb{N} \wedge B_\alpha\ne \phi)\\
\implies\exists \beta( \beta<\alpha \wedge a_\beta=x \wedge a_\beta \in \mathbb{N} \wedge B_\alpha\ne \phi)\\
\implies\exists \beta( \beta<\alpha \wedge a_\beta=x \wedge a_\beta = \cap B_\beta \wedge B_\alpha\ne \phi)\\
\implies\exists \beta( a_\beta=x \wedge a_\beta = \cap B_\beta \wedge B_\alpha\ne \phi \wedge\beta<\alpha)\\
\implies \exists \beta( a_\beta=x \wedge a_\beta = \cap B_\beta \wedge \cap B_\beta\le \cap B_\alpha)\\
\implies \exists \beta( a_\beta=x \wedge a_\beta \le \cap B_\alpha)\\
\implies x \le \cap B_\alpha$\\
Therefore, we obtain formula \ref{eq8}.
\end{proof}
As a result we have theorem \ref{theorem0}.
\begin{theorem}
\label{theorem0}
The greatest element of $\mathbb{N}$ exists.
\end{theorem}
\begin{proof}
From the property \ref{p3} of lemma \ref{lemma0} we know there is an ordinal number $\gamma$ such that $B_\gamma$ contains only one element $z$\\
$B_\gamma=\{z\}$\\
Considering definition \ref{def1}, we have\\
$B_\gamma=\mathbb{N}\backslash A_\gamma \implies B_\gamma \subseteq \mathbb{N} \implies z \in \mathbb{N}$\\
and\\
$B_\gamma=B_\gamma \cap \mathbb{N}=\mathbb{N}\backslash (A_\gamma\cap \mathbb{N})\to B_\gamma \cup (A_\gamma\cap \mathbb{N})=\mathbb{N}$\\
Notice $B_\gamma=\{z\}\neq \phi$ and formula \ref{eq8}. Then we have\\
$\forall x(x\in A_\gamma \cap \mathbb{N} \wedge B_\gamma\neq \phi\to x\le \cap B_\gamma)\\
\implies \forall x(x\in A_\gamma \cap \mathbb{N} \to x\le z)\\
\implies \forall x(x\in A_\gamma \cap \mathbb{N} \to x\le z) \wedge \forall x(x\in B_\gamma \to x \le z)\\
\implies \forall x(x\in A_\gamma \cap \mathbb{N} \vee x\in B_\gamma\to x\le z) \\
\implies \forall x(x\in (A_\gamma \cap \mathbb{N}) \cup B_\gamma\to x\le z) \\
\implies \forall x(x\in \mathbb{N} \to x \le z)$\\
So $z$ is greater than or equal to all the elements of $\mathbb{N}$. Noticing $z \in \mathbb{N}$, $z$ is the greatest element of $\mathbb{N}$. Therefore, the set $S$ defined in Eq. (\ref{eq3}) is not an empty set and equals $\{z\}$.
\end{proof}
Then we obtain theorem \ref{theorem1}.
\begin{theorem}
\label{theorem1}
$\mathbb{N}$ is an element of itself.
\end{theorem}
\begin{proof}
From the definition of $\mathbb{N}$ we know\\
$x\in \mathbb{N} \implies x^+\in \mathbb{N} \wedge x\in x^+ \implies x\in \cup \mathbb{N}$\\
so\\
$\forall x(x\in \mathbb{N}\to x\in\cup \mathbb{N})$\\
therefore\\
$\mathbb{N}\subseteq \cup \mathbb{N}$\\
As set $\mathbb{N}$ is transitive \cite{A}, we also have\\
$x\in \cup \mathbb{N} \implies \exists y(y\in \mathbb{N}\wedge x\in y)\implies x\in \mathbb{N}$\\
so\\
$\forall x(x\in \cup \mathbb{N} \to x\in \mathbb{N})$\\
therefore\\
$\cup \mathbb{N}\subseteq \mathbb{N}$\\
Considering above two cases we obtain
\begin{equation}
\label{eq10}
\cup \mathbb{N}=\mathbb{N}.
\end{equation}
Considering theorem \ref{theorem0} $z$ is the greatest element of $\mathbb{N}$ and Eq. (\ref{eq10}), we have\\
$x \in z \implies x \in z\wedge z\in \mathbb{N}\implies x \in \cup \mathbb{N} \implies x \in \mathbb{N}$\\
so\\
$\forall x(x \in z \to x \in \mathbb{N})$\\
and\\
$z\subseteq \mathbb{N}$\\
On the other hand\\
$x\in \mathbb{N}\implies x\in \cup \mathbb{N}\implies \exists y(y\in \mathbb{N} \wedge x\in y\wedge y\le z)\\
\implies \exists y(y\in \mathbb{N} \wedge x\in y\wedge y\subseteq z)\implies x\in z$\\
so\\
$\forall x(x\in \mathbb{N}\to x \in z)$\\
and\\
$\mathbb{N} \subseteq z$\\
Considering above two cases we obtain\\
$z=\mathbb{N}$\\
and
\begin{equation}
\label{eq11}
\mathbb{N}\in \mathbb{N}.
\end{equation}
\end{proof}

However, the conclusion of formula (\ref{eq11}) that $\mathbb{N}$ is the greatest element of itself not only conflicts with the common sense that there is no greatest natural number, but more severely it contradicts the axiom of regularity which asserts a set cannot be a member of itself \cite{A}. And the latter is a serious conflict, because it leads to conflict between the two axioms of set theory.


\section{Discussion}
The most important part of this paper is the proof of the property \ref{p3} of lemma \ref{lemma0} with which some people may disagree. The most common argument against it is based on the theory of limit ordinals. They argue that from the formula $\forall \alpha(\alpha<\lambda \to B_{\alpha+1}\neq\phi)$ we cannot deduce $B_\lambda \neq \phi$ when ordinal $\lambda$ is a limit ordinal. They maintain, on the contrary, actually $B_\lambda=\phi$ and $\lambda=\omega$ where $\omega$ is the first transfinite limit number.

However, there is some difficulty in this argument neglected by the defenders of set theory. That is if $\lambda$ is a limit ordinal the proposition that the iterative process cannot deduct all the elements of $\mathbb{N}$ before step $\lambda$ is ambiguous in the logic system of set theory. On the one hand, based on definition \ref{def3} we know the expression that the iterative process cannot deduct all the elements of $\mathbb{N}$ refers to that the iterative process cannot deduct all the elements $\mathbb{N}$ by every ordinal step. As a result, the proposition that the iterative process cannot deduct all the elements of $\mathbb{N}$ before step $\lambda$, a restriction version of above expression with a restriction to limit its concerning domain to the steps only before step $\lambda$ (it does not care about whether the iterative process can deduct all the elements by or after step $\lambda$ for it does not mention this in the statement, and it just care about the results of the steps before $\lambda$, so whether or not the proposition is true is only decided by the results of the steps before $\lambda$), should, as its literal meaning, be interpreted as the iterative process cannot deduct all the elements of $\mathbb{N}$ by every step before $\lambda$. Then the interpretation can be written as $\forall \alpha(\alpha<\lambda \to B_{\alpha+1}\neq\phi)$. This interpretation expresses the original meaning of the proposition. On the other hand, according to the explanation of definition \ref{def1} we know $B_\lambda $ is the particular remaining part of $\mathbb{N}$ exactly before step $\lambda$ is performed. So the proposition is also logically interpreted as $B_\lambda\ne\phi$ (the remaining part is still not empty before step $\lambda$ is performed). The second interpretation expresses the extended meaning of the proposition based on, $B_\lambda$, the critical result of exactly before step $\lambda$. As both interpretations are correct in semantics they should be equivalent in logic. And for most instances, i.e. if $\lambda$ is a successor ordinal the two forms of interpretation are equal indeed so they do not cause any logic problem. But if $\lambda$ is a limit ordinal,
the two formulas, $\forall \alpha(\alpha<\lambda \to B_{\alpha+1}\neq\phi)$ and $B_\lambda\neq\phi$, are not equal within the framework of set theory, so the proposition is ambiguous in the logic system which contains set theory. And this ambiguity will jeopardize the rigorousness and consistency of the logical system and will lead to contradiction when considering the question of whether the iterative process can deduct all the elements of $\mathbb{N}$ before step $\lambda$ with the assumption $B_\lambda=\phi$. Therefore even $\lambda$ is a limit ordinal does not solve all the problems.

Further analysis at the level of general logic reveals that the notion of transfinite limit ordinal is the source of all difficulties. In set theory formula $\forall x\in C(F(x))$ means every element of class $C$ has the property $F$. So we can check some elements of $C$, and if we have already checked that all the elements of $C$ have the property $F$ we will be fully convinced that the formula is true. If $C$ is a finite class the verification of the formula must be conducted in this complete way since having all the elements of $C$ checked is possible. But if $C$ is not a finite class, the situation is quite different. In this case we are not sure whether having all the elements of $C$ checked is possible in principle for it involves whether an infinite process of checking can complete (from the point of view of potential infinity, an infinite process cannot complete, but in the light of set theory some infinite processes do can complete which is analyzed in the following section of this paragraph). So the safe scheme to this situation should be that if having all the elements of $C$ checked is possible we should check all the elements before coming to the conclusion that the formula is true and if it is no possible we need not complete the check of all elements and we just need check every single element to ensure it has the property $F$ (obviously, under such circumstance, the check process cannot complete otherwise having all the elements checked is possible). So having every single element checked is possible does not mean having all elements checked is possible, and we call the former a weak check and the latter a strong check. A weak check just means a infinite process of checking but does not guarantee the process can complete. If a weak check cannot complete indeed the strong check is impossible and also does not exist. If a weak check can complete it turns out to be a strong check. With these notions we come to an important conclusion that if class $C$ is the ordinal number $\lambda$, the formula $\forall \alpha\in \lambda(F(\alpha))$ must be verified with a strong check. Heuristically, an iterative process can be used to check whether every element $\alpha$ of $\lambda$ has the property $F(\alpha)$, i.e. at step 0 it checks $F(0)$, at step 1 it checks $F(1)$ and so on. So exactly before step $\lambda$ it has already checked all the elements in $\lambda$ that indicates the the check process has completed before step $\lambda$. As a result, the infinite check process does can complete if $C$ is an ordinal. Another important and interesting conclusion is that if $C$ is a proper class the formula $\forall \alpha\in C(F(\alpha))$ must be checked with a weak check. In that case $C$ is too big to have any definite cardinality so it is not possible to complete the check of all elements of $C$ within any ordinal step (within any ordinal step we just can check a portion of $C$ with a definite cardinality, in set theory the checked portion is called a subset of $C$, but we never can obtain a checked portion that exactly covers the whole domain of $C$ which is too big to be covered by any static completed entity, i.e. any set). Therefore, for a proper class, a strong check is impossible and the infinite check process is incompletable in essence. If $C$ is the proper class $On$, we usually use transfinite induction to prove or check formula $\forall \alpha\in On(F(\alpha))$. That is if $F(0)$ holds and $\forall \beta<\alpha(F(\beta))\to F(\alpha)$ holds then $\forall \alpha\in On(F(\alpha))$ holds. Under such checking scheme, it is obvious that there is no ordinal number in $On$ which is not checked by the transfinite induction. But from this we should not go too far to infer that the transfinite induction has already checked all the ordinal numbers in $On$ that contradicts above second conclusion that the infinite check process of a proper class is incompletable. So we are in a nuanced situation which suggests the law of excluded middle is not applicable to proper class and the two facts (having all ordinal numbers in $On$ checked is impossible and there is no ordinal number in $On$ that is not checked yet) must both hold. This argument also explains some intuitionists' concern about the abuse of the law of excluded middle to infinite set. To be exact, here it should be proper class rather than infinite set. And the secret to the question lies in the fact that a proper class is not a completed entity so it is not possible to complete the check of all its elements although the check of any element in it is completable.

As $\lambda$ is an ordinal number, it is the exact set of all the steps before $\lambda$. In the light of the implementation of an iterative checking process, every step $\alpha$ in $\lambda$ can be checked, and exactly checking step $\alpha$ means step $\alpha$ and all steps before it are checked but steps after step $\alpha$ are not checked yet. If $\lambda$ is a transfinite limit ordinal, it has the property $\forall \alpha\in \lambda(\alpha+1 \in \lambda)$ which means every step $\alpha$ in $\lambda$ is not the end step in $\lambda$. As a result, the infinite check process of all the steps in $\lambda$ also does not have an end step. And the essence of limit ordinal theory lies in it tries to make us believe that a check process without an end step can end. Integrating with the implementation of an iterative checking process, the formula can be given further interpretation. It means for every step $\alpha$ in $\lambda$ when it has been checked the event that all and only all steps in $\lambda$ are checked does not happen for step $\alpha+1$ is in $\lambda$ and not checked yet. So the critical event does not happen at any step before step $\lambda$. But, on the other hand, with the checking regulation we know at step $\lambda$ it begins to check step $\lambda$ that implies all steps in $\lambda$ are already checked before step $\lambda$ is checked, so the critical event does happen before step $\lambda$. Therefore, in the context of set theory, we have to draw a peculiar conclusion that although the critical event does not happen at any step before step $\lambda$ it does happen before step $\lambda$. And we are extremely curious about why mathematicians can not find any logical flaw in the conclusion. If the conclusion holds there must be a mystery state which is exactly before step $\lambda$ but after all steps in $\lambda$ and the critical event, therefore, does happen at such mystery state. Unfortunately in the context of set theory there is no such state for the limit ordinal number $\lambda$ does not have an immediate predecessor. So the notion of transfinite limit ordinal is illogical and it is not a proper solution to the problem.

Form above discussion and the derivation of formula (\ref{eq11}) we know that if we insist on $\mathbb{N}$ is a set we must be confronted with either a contradiction or a conflict with the axiom of regularity. Both of them are deadly to set theory.

And here we cannot solve the problem by sacrificing the axiom of regularity. If we do so, Eq. (\ref{eq1}) should be revised as following completed form to satisfy formula (\ref{eq11}) regardless of the violation of regularity
\begin{equation}
\mathbb{N}=\{0,1,2,\cdots,\mathbb{N}\}.
\end{equation}
This form of definition of $\mathbb{N}$, however, is impredicative \cite{B} and contains a vicious circle \cite{C}, from which we even cannot determine the exact value of $\mathbb{N}$ since $\mathbb{N}$ appears in both sides of the definition. And what is more, without regularity we even cannot prevent Mirimanoff's paradox \cite{D}. Therefore this scheme is totally unacceptable, and the axiom of infinity should be excluded from set theory to keep the theory consistent.

Since the class of all natural numbers defined by the comprehension principle in Eq. (\ref{eq2}) cannot be a set, in the light of NBG set theory \cite{E}, it should be a proper class. The essence of $\mathbb{N}$ is its incompleteness and non-substantiality. In other words $\mathbb{N}$ is too large to be any completed entity, and it just can be a dynamic class which is always under construction. Weyl had obviously seen the difference between completed entity and dynamic class, and deemed that blindly converting one into the other is the true source of our difficulties and antinomies, a source more fundamental than Russell's vicious circle principle indicated \cite{F}. Our work has made it clear that the dynamic class $\mathbb{N}$ cannot be a set for its incompleteness, and also discloses the essential difference between set and proper class that is obscure in set theory.

If we do not regard $\mathbb{N}$ as a set but a proper class all the difficulties we encounter in this paper will be resolved. That is if $\mathbb{N}$ is not a set we cannot prove $B$ and $A$ are sets. So there is no the conclusion that $A$ is an initial segment of ordinal and also there is no the ordinal number $\lambda$. Consequently we cannot obtain the property \ref{p3} of lemma \ref{lemma0}; as a result, the proofs of theorem \ref{theorem0} and \ref{theorem1} are groundless. And the first contradiction is also dismissed in the absence of the ordinal number $\lambda$.

\section{Conclusion}
The difficulties reveal that the axiom of infinity which guarantees the existence of the set of natural numbers causes either a contradiction or a conflict with the axiom of regularity, and the essence of the contradiction lies in the inductive definition of set $\mathbb{N}$. When we define the inductive collection $\{0, 1, 2, \cdots\}$ produced by the inductive add-one process is an infinite set $\mathbb{N}$, we have already regarded it as a completed, static entity. But on the other hand, with regularity and the induction principle, the inductive construction of natural numbers still can step into the next step wherever it attains and produces a new natural number. So the completed state of the inductive construction does not exist that implies the infinite set $\mathbb{N}$ also does not have a completed form. How can an already existing entity possess the attribute that it does not have a completed form at the same time? This is the insidious logical fallacy deeply hiding behind the axiom of infinity.

In our point of view the inductive definition of natural numbers just could guarantee the existence of an infinite process, but it should not become the sufficient condition for that the infinite process can be finally done and thus produce an infinite static totality, i.e., an infinite set. That is the misapprehension of infinity in the notion of actual infinity.

Since we have proved that the class of all natural numbers cannot be a set, the assertion made in the axiom of infinity that there is an inductive set is improper.

\emph{Acknowledgments}. The author is grateful to Kyrill for carefully reading the manuscript and making useful suggestions.

\end{document}